\documentclass{amsart}
\usepackage{amsfonts}
\usepackage{mathrsfs}  
 \usepackage{amsmath}    
 \usepackage{amsthm}     
 \usepackage{amscd}      
 \usepackage{amssymb}    
 \usepackage{eucal}      
 \usepackage{latexsym}   
 \usepackage{graphicx}   
 \usepackage{verbatim}   
\usepackage[all]{xy}   

\newcommand{\nN}{\mathbb{N}}                     
\newcommand{\nZ}{\mathbb{Z}}                     
\newcommand{\nR}{\mathbb{R}}                     
\newcommand{\nC}{\mathbb{C}}                     

\newcommand{\nP}{\mathbb{P}}                     

\newcommand{\sE}{\mathscr{E}}
\newcommand{\sF}{\mathscr{F}}
\newcommand{\sO}{\mathscr{O}}                    
\newcommand{\sH}{\mathscr{H}}
\newcommand{\sI}{\mathscr{I}}                    

\newcommand{\sK}{\mathscr{K}}

\DeclareMathOperator{\bl}{Bl}                    
\DeclareMathOperator{\bireg}{bireg}              

\DeclareMathOperator{\sHom}{\mathscr{H}om}       
\DeclareMathOperator{\Hom}{Hom}                  
\DeclareMathOperator{\Supp}{Supp}                
\DeclareMathOperator{\reg}{reg}                  
\DeclareMathOperator{\supp}{Supp}                
\DeclareMathOperator{\Proj}{Proj}                

\newtheorem{theorem}{Theorem}[section]
\newtheorem{proposition}[theorem]{Proposition}
\newtheorem{lemma}[theorem]{Lemma}
\newtheorem{corollary}[theorem]{Corollary}

\theoremstyle{definition}
\newtheorem{definition}[theorem]{Definition}
\newtheorem{remark}[theorem]{Remark}

\newtheorem{example}[theorem]{Example}

\numberwithin{equation}{section}

\begin{document}
\title[Some Results on Asymptotic Regularity of Ideal Sheaves]{Some Results on Asymptotic Regularity of Ideal Sheaves}

\author{Wenbo Niu}
\address{Department of Mathematics, Statistics, and Computer
Science, University of Illinois at Chicago, 851 South Morgan Street,
Chicago, IL 60607-7045, USA}

\email{wniu2@uic.edu}



\date{}
\maketitle
\begin{abstract} Let $\sI$ be an ideal sheaf on $\nP^n$ defining a subscheme $X$. Associated to $\sI$ there are two elementary invariants: the invariant $s$ which measures the positivity of $\sI$, and the minimal number $d$ such that $\sI(d)$ is generated by its global sections. It is now clear that the asymptotic behavior of $\reg \sI^t$ is governed by $s$ but usually not linear. In this paper, we first describe the linear behavior of the asymptotic regularity by showing that if $s=d$, i.e., $s$ reaches its maximal value, then for $t$ large enough $\reg \sI^t=dt+e$ for some positive constant $e$. We then turn to concrete geometric settings to study the asymptotic regularity of $\sI$ in the case that $X$ is a nonsingular variety embedded by a very ample adjoint line bundle. Our approach also gives regularity bounds for $\sI^t$ once we know $\reg \sI$ and assume that $X$ is a local complete intersection.
\end{abstract}

\tableofcontents
\section{Introduction}

\noindent The motivation of this paper is to understand several interesting phenomena arose in recent research on the asymptotic Castelnuovo-Mumford regularity of an ideal sheaf. Throughout the paper, we work over the field of complex number $\nC$ and a variety is always reduced and irreducible.

Let us first mention an interesting result on the asymptotic regularity of a homogeneous ideal, for which the picture seems rather clear now thanks to the effort of many people. Suppose that $I$ is a homogeneous ideal in a polynomial ring. Swason \cite{Swanson:PowOfIdelas} first observed that for $t$ sufficiently large $\reg I^t$ is bounded by a linear function $dt+e$ for some constant $d$ and $e$. An effective result was soon established independently by Kodiyalam \cite{Kodiyalam:AymReg} and Cutkosky, Herzog  and Trung \cite{Cutkosky:AsymReg}, which says that for $t$ sufficiently large one actually can get $$\reg I^t=dt+e.$$ The slope $d$ in the formula has a concrete algebraic meaning in the result. However much less is known about the constant $e$ in the equality even recently.

Turning to the geometric case, suppose that $\sI$ is an ideal sheaf on the projective space $\nP^n$. The shape of the asymptotic regularity of $\sI$ was first described by Cutkosky, Ein and Lazarsfeld \cite{Ein:PosiComlIdeal} in the formula $$\lim_{t\rightarrow \infty}\frac{\reg\sI^t}{t}=s.$$
The number $s$ is called the $s$-invariant of $\sI$ and measures, roughly speaking, the positivity of $\sI$. Thus this formula gives us an impression that the positivity of an ideal sheaf controls its asymptotic regularity. Several examples in \cite{Cutkosky:AsymReg} and \cite{Ein:PosiComlIdeal} have shown that $s$ could be an irrational number which prevents the asymptotic regularity from being linear.  In general, the best one can hope for, is that the asymptotic regularity is bounded by linear functions \cite{Niu:AVanishing}. Since usually the Rees algebra $\oplus \sI^t$ is not finitely generated it seems very hard that the asymptotic regularity is linear. However a recent work of Chardin \cite{Chardin:Powersofideals} showed a surprising result that $\reg \sI^t$ could be linear for $t$ sufficiently large.

In order to understand Chardin's result and put those linear and nonlinear phenomena together in a clear picture, we give the following linearity theorem of asymptotic regularity.
\begin{theorem}\label{intr:01} Let $d$ be the minimal number such that $\sI(d)$ is generated by its global sections and $s$ be the $s$-invariant of $\sI$. If $s=d$  then for $t$ large enough,  one has $\reg \sI^t=dt+e$ for some constant $e\geq 0$.
\end{theorem}
\noindent The meaning of the constant $e$ can be explained in terms of the relative regularity and the biregularity of the blowing-up of $\nP^n$ along $\sI$ (for more details see Section 2). Thus this theorem strengthens Chardin's result and justifies our intuition that positivity controls asymptotic regularity.

As a quick corollary of the theorem, we give a typical variety having linear asymptotic regularity as follows. It is not so obvious and also tells us that varieties having linear asymptotic regularity are in fact not so rare.
\begin{corollary} Suppose that $\sI$ defines a (nonlinear) variety cut out by quadrics, then its asymptotic regularity is linear, i.e., for $t\gg 0$, one has
$\reg \sI^t=2t+e,$
for some positive integer $e\geq 0$.
\end{corollary}

As another application of our linearity theorem, we can slightly strengthen a theorem due to Vermiere \cite{Vermeire:RegPowers} which gives the asymptotic regularity of a curve embedded by a line bundle of large degree. In this situation, the constant $e$ can be determined explicitly as either $0$ or $1$ depending on the surjectivity of the corresponding Gauss map (see Proposition \ref{pro:05}).

Let us further assume that $\sI$ defines a projective subvariety  $X$ in $\nP^n$. The embedding of $X$ is determined by a very ample line bundle $L$ on $X$. The syzygies of $X$ therefore heavily depend on the positivity of $L$. An efficient way to describe such embedding is using Green's condition $(N_p)$. We hope that knowing the syzygy of $X$ through $(N_p)$ condition could give information about the asymptotic regularity of $X$.

The first case when $X$ is a nonsingular projective curve has been investigated by Vermiere \cite{Vermeire:RegPowers}. He showed that if $X$ is embedded by $L$ of degree $\geq 2g+3$, where $g$ is the genus of $X$, then $\sI^t$ is $(2t+1)$-regular for $t\geq 1$. According to the theorem of Green and Lazarsfeld \cite{Lazarsfeld:SyzFiniteSet} such $L$ satisfies at least  Property $(N_2)$. For higher dimensional nonsingular projective variety $X$, we consider the case that the embedding is determined by the following  adjoint line bundle
$$L_d=K_X+dA+B,$$
where $K_X$ is the canonical bundle of $X$, $A$ is a very ample line bundle and $B$ is a nef line bundle. A theorem of Ein and Lazarsfeld \cite{Ein:SyzygyKoszul} shows that if $d\geq \dim X+1+p$, then $L_d$ satisfies Property $(N_p)$. Thus an interesting question is up to which extend of the positivity of $L_d$ one could know the asymptotic regularity of $X$.

The crucial point to get such asymptotic regularity bounds is to establish a vanishing theorem of the tensor algebra of the conormal bundle of $X$ (see Section 3). This idea is quite straightforward and has been used by Vermeire in the aforementioned work. As a general result we show the following theorem by assuming that $\dim X\geq 2$, that for the sufficient positive line bundle $L_d$ we can get regularity bounds for powers of an ideal sheaf.

\begin{theorem} Assume that $d\geq 2(\dim X+1)$ in the adjoint line bundle $L_d$. Then for any $t\geq 1$, the ideal sheaf $\sI^t$ is $(2t+2\dim X-2)$-regular.
\end{theorem}
By the result of Ein and Lazarsfeld, such $L_d$ in the theorem satisfies Property $(N_{\dim X+1})$ and therefore $X$ is cut out by quadrics. Applying the linearity theorem \ref{intr:01} above, we see that the constant $e$ satisfies $0\leq e\leq 2\dim X-2$. In particular if $X$ is a surface, we have $0\leq e\leq 2$.

In the same way, we can also have an interesting result on regularity bounds for powers of $\sI$ if $X$ is a locally complete intersection (equidimensional) and $\sI$ is $r$-regular.

\begin{theorem} Assume that $X\subset \nP^n$ is a local complete intersection defined by $\sI$ and assume that $\sI$ is $r$-regular.
\begin{enumerate}
\item If $\dim X=1$, then for any $t\geq 1$, $\sI^t$ is $rt$-regular.
\item If $\dim X=2$, then for any $t\geq 1$, $\sI^t$ is $rt+r-2$-regular.
\item If $\dim X\geq 3$, then for any $t\geq 1$, $\sI^t$ is $rt+\max(r,(\dim X -1)r-\dim X)$ regular.
\end{enumerate}
\end{theorem}

This paper is organized as follows. We first prove the linearity theorem of asymptotic regularity in Section 2. Then we build a vanishing theorem for tensor products of the conormal bundle of a variety in Section 3. In Section 4 and 5 we apply our vanishing theorem to get regularity bounds for powers of ideal sheaves.

\vspace{0.5cm}

\noindent{\em Acknowledgement.} Special thanks are due to the author's advisor Lawrence Ein who offers a lot of help and suggestions. The author also thanks Marc Chardin and Pete Vermeire for their explanations on their results and useful discussions and suggestions.

\section{Linearity of asymptotic regularity}

\noindent In this section, we study the linearity of the asymptotic regularity of an ideal sheaf on the projective space. Throughout, we fix our notation as follows. Suppose that $\sI$ is an ideal sheaf on $\nP^n$. Let $d$ be the minimal number such that $\sI(d)$ is generated by its global sections. We denote by $H$ the hyperplane divisor of $\nP^n$. We shall use line bundle and divisor interchangeable if there is confuse likely. Consider the blowing-up
$$\mu: W=\bl_{\sI}\nP^n\longrightarrow \nP^n, $$
with an exceptional divisor $E$ such that $\sI\cdot
\sO_W=\sO_W(-E)$. We denote by $V=H^0(\sI(d))$ the vector space of global sections of $\sI(d)$. Those sections determine a surjective morphism
$$V\otimes\sO_{\nP^n}(-d)\longrightarrow \sI\longrightarrow 0.$$
Via the surjective morphism from the symmetric algebra of $\sI$ to its Rees algebra, the morphism above gives the embedding of the blowing-up $W$ in the biprojective space
$$W\hookrightarrow \nP(V\otimes\sO_{\nP^n}(-d))\cong\nP^n\times\nP(V).$$
Note that $\sO_W(d\mu^*H-E)$ is generated by its global sections. Let $p$ and $q$ be projections of $\nP^n\times \nP(V)$ to its components, we then have the diagram
\begin{equation}\label{eq:9}
\xymatrix{
W=\bl_{\sI}\nP^n \ar[dr]^{\mu} \ar@/^1.5pc/[rr]|\pi \ar@{^{(}->}[r] & \nP^n\times \nP(V) \ar[d]^p \ar[r]^q & \nP(V) \\
  & \nP^n.}
\end{equation}
On the biprojective space $Y=\nP^n\times \nP(V)$, for any coherent sheaf $\sF$, we denote by $$\sF(a,b)=\sF\otimes p^*\sO_{\nP^n}(a)\otimes q^*\sO_{\nP(V)}(b).$$ Under such notation, we see that
\begin{equation}\label{eq:10}
\sO_Y(0,1)|_W=\sO_W(0,1)=\pi^*\sO_{\nP(V)}(1)=\sO_W(d\mu^*H-E).
\end{equation}

Now we recall some basic definitions. The first one is the $s$-invariant of $\sI$, which measures its positivity. See \cite[Section 5.4]{Lazarsfeld:PosAG1} for details.

\begin{definition} The {\em $s$-invariant} of
$\sI$ (with respect to the divisor $H$) is the positive real number
$s(\sI)=\min\{\ s\ |\ s\mu^*H-E \ \mbox{ is nef }\}$.
Here $ s\mu^*H-E$ is considered as an $\nR$-divisor on $W$.
\end{definition}

It is easy to see that $s\leq d$ since $\sI(d)$ is generated by its global section and then the divisor $\sO_W(d\mu^*H-E)$ is nef.

There are two generalized notions of regularity: biregularity on a biprojective space and relative regularity on a projective bundle.

\begin{definition}\label{def:1} Let $\sF$ be a coherent sheaf on a biprojective space $Y=\nP^a\times\nP^b$. We say that $\sF$ is $\textbf{m}=(m_1,m_2)$-regular if
$$H^i(Y,\sF(m_1-u,m_2-v))=0$$
for all $i>0$ and $u+v=i$ where $(u,v)\in \nN^2$ (we assume $0\in \nN$).
\end{definition}

Denote by $\bireg \sF$ the set of the pair $\textbf{m}$ such that $\sF$ is $\textbf{m}$-regular. We also denote by $b_1(\sF)=\min\{m_1|\ \textbf{m}=(m_1,m_2)\in \bireg \sF\}$ and $b_2(\sF)=\min\{m_1|\ \textbf{m}=(m_1,m_2)\in \bireg \sF\}$. Note that $b_1$ and $b_2$ could be $-\infty$. The basic property of biregualrity we will use is that if $\sF$ is $\textbf{m}$-regular, then it is $\textbf{m}+\nN^2$-regular.

\begin{definition} Let $X$ be a variety and $\sE$ be a vector bundle on $X$, with the projectivization $\pi: \nP(\sE)\rightarrow X$. A coherent sheaf $\sF$ on $\nP(\sE)$ is $m$-regular with respect to $\pi$ if
$$R^i\pi_*(\sF\otimes\sO_{\nP(\sE)}(m-i))=0$$
for $i>0$.
\end{definition}

Denote by $\reg_{\pi}\sF$ the minimal number $m$ such that $\sF$ is relative $m$-regular with respect to $\pi$. Note that $\reg_{\pi}\sF$ could be $-\infty$. If $\sF$ is relative $m$-regular, then it is relative $(m+1)$-regular.

\begin{proposition} Let $X$ be a variety and $\sE$ be a locally free sheaf with the projectivization $\pi:\nP(\sE)\rightarrow X$.  Suppose that $\sF$ is a coherent sheaf on $\nP(\sE)$ and let $c$ be a nonnegative integer. Then
$R^i\pi_*\sF(k)=0\mbox{ for all }i>c, k\in \nZ$ if and only if for any $x\in X$, the restriction $\sF_x$ of $\sF$ on the fiber over $x$ satisfies $\dim \supp\sF_x\leq c$.
\end{proposition}
\begin{proof} The sufficient part is clear. We prove the necessary part, i.e., assume that $R^i\pi_*\sF(k)=0\mbox{ for all }i>0, k\in \nZ$.

The plan is to use the formal function theorem and prove by contradiction. Suppose that there is a point $x\in X$ such that $P_x=\pi^{-1}(x)$ is the fiber and $d=\dim \Supp\sF_x>c$. We claim that for  $k\gg 0$,
\begin{equation}\label{eq:11}
H^d(P_x,\sF_x(-k))\neq 0.
\end{equation}
For this, let $V=\Supp\sF_x\subset P_x$ and we give a reduced scheme structure to $V$. Then it is enough to show that for $k\gg 0$,
$$H^d(V,\sF_x|_V(-k))\neq 0.$$
Because from the exact sequence $0\rightarrow \ker\rightarrow \sF_x\rightarrow \sF_x|_V\rightarrow 0$, we see that the sections of $H^d(V,\sF_x|_V(-k))$ will be lifted to $H^d(P_x,\sF_x(-k))$ since $\dim \Supp \ker\leq d$. Let $\omega_V$ be the dualizing sheaf of $V$. By duality, one has
$$H^d(V,\sF_x|_V(-k))=\Hom(\sF_x|_V,\omega_V(k))=H^0(V,\sHom(\sF_x|_V,\omega_V)(k)).$$
Thus it is enough to show that the sheaf $\sHom(\sF_x|_V,\omega_V)\neq 0$. This is then a local question. We may shrink $V$ if necessary to assume that $V$ is nonsingular and $\omega_V=\sO_V$. Thus locally we have a surjective morphism $\oplus \sO_V\rightarrow (\sF_x|_V)^{\vee}\rightarrow 0$ which gives an injection $(\sF_x|_V)^{\vee}\hookrightarrow \oplus\sO_V$. Also note that $\sF_x|_V$ is not torsion sheaf on $V$. By the generic flatness, the natural morphism $\sF_x|_V\rightarrow (\sF_x|_V)^{\vee}$ is nonzero. Composing it with the injection above, we have a nonzero morphism $\sF_x|_V\rightarrow \oplus\sO_V$. Thus by choosing to project to $\sO_V$, we obtain a nonzero morphism $\sF_x|_V\rightarrow \sO_V$. Thus $\sHom(\sF_x|_V,\sO_V)\neq 0$ and thus (\ref{eq:11}) is true.

Now let $nP_x$ is the scheme defined by $\sI^n_{P_x}$, which is the $n$-th thickening of $P_x$, and let $\sF_{nx}$ be the restriction of $\sF$ to $nP_x$. Now from the sequence
$$0\rightarrow \ker\rightarrow \sF_{(n+1)x}\rightarrow \sF_{nx}\rightarrow 0$$
and the fact that $\dim\supp \sF_x=d$, we see that any nonzero section of $H^d(P_x,\sF_x(-k))$ in (\ref{eq:11}) will be lifted to $H^d(P_x,\sF_{(n+1)x}(-k))$. Thus by the formal function theorem, we have that $\hat{R}^d\pi_*\sF(-k)\neq 0$ which is contradict to the assumption.
\end{proof}

Taking $c=0$ in the proposition, we get the following corollary.

\begin{corollary}\label{pro:10} Keep notation as in Proposition above. $R^i\pi_*\sF(k)=0\mbox{ for all }i>0, k\in \nZ$ if and only if for any $x\in X$, the restriction $\sF_x$ of $\sF$ on the fiber over $x$ satisfies $\dim \supp\sF_x\leq 0$.
\end{corollary}

Consider a coherent sheaf $\sF$ on a biprojective space $Y=\nP^a\times \nP^b$. We denote by $p_1$ and $p_2$ the projections of $Y$ to its components. Then $\sF$ has biregularity and also has relative regularity with respect to $p_1$ and $p_2$. In the following theorem, we relate these two types of regularity together.

\begin{proposition}\label{pro:13} Let $\sF$ be a coherent sheaf on a biprojective space $\nP^a\times \nP^b$ with $p_1$ and $p_2$ the projection morphisms to its components. Let $r$ be an integer. Then $\sF$ is relative $r$-regular with respect to $p_2$ if and only if $\sF$ is $(r,r')$-biregular for some integer $r'$.
\end{proposition}
\begin{proof} We view $Y=\nP^a\times \nP^b$ as a projectivized vector bundle over $\nP^b$ with the tautological line bundle $\sO_Y(1)=p^*_1\sO_{\nP^a}(1)$. Denote by $\delta=\dim Y$. Suppose that $\sF$ is relative $r$-regular with respect to $p_2$, i.e., $R^i{p_2}_*\sF(r-i)=0$ for all $i>0$. For any integers $k_1$ and $k_2$ we have a spectral sequence
$$E^{p,q}_2=H^p(\nP^b,R^qp_{2*}\sF(k_1)\otimes \sO_{\nP^b}(k_2)))\Rightarrow H^{p+q}(Y, \sF(k_1)\otimes p^*_2\sO_{\nP^b}(k_2)).$$
By Serre's vanishing theorem, there exists a number $n_0$ such that for all $k_2>n_0$ and for all $k_1=r,r-1,\cdots, r-\delta$, $$H^p(\nP^b,R^q_{p_2*}\sF(k_1)\otimes \sO_{\nP^b}(k_2))=0,\quad\mbox{for all } p>0,q\geq 0.$$ We then deduce from the spectral sequence that for $k_2>n_0$ and $k_1=r,r-1,\cdots, r-\delta$,
\begin{equation}\label{eq:12}
H^0(\nP^b,R^i_{p_2*}\sF(k_1)\otimes \sO_{\nP^b}(k_2))=H^i(Y,\sF(k_1)\otimes p^*_2\sO_{\nP^b}(k_2)),
\end{equation}
Now for each $0<i\leq \delta$, since $\sF$ is relative $r$-regular with respect to $p_2$, we see that
$$R^ip_{2*}(\sF(r-u))=0,\quad \mbox{for }0\leq u\leq i.$$
Thus using (\ref{eq:12}) with $k_1=r-u$, we see
$$H^i(Y,\sF(r-u)\otimes p^*_2\sO_{\nP^b}(k_2))=0,\quad\mbox{for all } k_2\geq n_0.$$
Then taking $r'>n_0+\delta$, we see that
$$H^i(Y,\sF(r-u,r'-v))=0,$$
where $(u,v)\in \nN^2$ and $u+v=i$.  This shows that $\sF$ is $(r,r')$-biregular.

Now suppose that $\sF$ is $(r,r')$-biregular. Then $\sF$ is $(r,k_2)$-biregular for all $k_2\geq r'$. Using spectral sequence above, for each $0\leq i\leq \delta$ and $k_2\gg 0$, we have
$$H^0(\nP^b,R^ip_{2*}\sF(r-i)\otimes \sO_{\nP^b}(k_2))=H^i(Y,\sF(r-i)\otimes p^*_2\sO_{\nP^b}(k_2)).$$
Since by the biregularity assumption that $H^i(Y,\sF(r-i)\otimes p^*_2\sO_{\nP^b}(k_2))=H^i(Y,\sF(r-i,k_2))=0$, we immediately have $R^ip_{2*}\sF(r-i)=0$. Thus $\sF$ is relative $r$-regular with respect to $p_2$.
\end{proof}

\begin{corollary}\label{pro:14} Keep notation as in Proposition \ref{pro:13}. Denote by $b_1(\sF)=\min\{m_1|\ \textbf{m}=(m_1,m_2)\in \bireg \sF\}$. Then one has
$$b_1(\sF)=\reg_{p_2}\sF.$$
\end{corollary}
Note that in the corollary, we do allow $b_1(\sF)$ and $\reg_{p_2}\sF$ to be $-\infty$. In fact, using Corollary \ref{pro:10}, we can easily summarize in the following corollary the case when these numbers are $-\infty$.

\begin{corollary}\label{pro:12} Keep notation as above. The following are equivalent
\begin{enumerate}
\item $b_1(\sF)=-\infty$;
\item $\reg_{p_2}\sF=-\infty$;
\item $\dim\Supp \sF_x\leq 0$ for any $x\in \nP^b$.
\end{enumerate}
\end{corollary}

Going back to the picture at the beginning. Think of $\sO_W$ as a coherent sheaf on the biprojective space $Y=\nP^n\times \nP(V)$, then it has biregularity and therefore the number $b_1(\sO_W)$ as in the definition (\ref{def:1}) is defined. The sheaf $\sO_W$ also has relative regularity $\reg_{q}\sO_W$ with respect to $q$. The relation between these notions of regularity and the positivity of $\sI$ is described in the following proposition.

\begin{proposition}\label{pro:15} The following are equivalent
\begin{enumerate}
\item $b_1(\sO_W)=-\infty$;
\item $\reg_{q}\sO_W=-\infty$;
\item $s<d$;
\item $\pi:W\rightarrow \nP(V)$ is finite.
\end{enumerate}
\end{proposition}
\begin{proof} (1) $\Leftrightarrow$ (2) $\Leftrightarrow$ (4) is from Corollary \ref{pro:12}. (3) $\Leftrightarrow$ (4) is from the definition of $s$-invariant.
\end{proof}

Now we come to our main theorem of this section. We show that if $s=d$ then the asymptotic regularity of $\sI$ will be linear. We also give the way to compute the constant part of the asymptotic regularity; it is the relative regularity of the blowing-up $\sO_W$.
\begin{theorem}\label{thm:02} Suppose that $s=d$ and denote by $e=b_1(\sO_W)=\reg_{q}\sO_W$ (the equality is guaranteed by Corollary \ref{pro:14}). For $t$ large enough,  one has
$$\reg \sI^t=dt+e,$$
and $e\geq 0$. Thus we see that for the ideal sheaf $\sI$ with maximal positivity, it has linear asymptotic regularity.
\end{theorem}
\begin{proof} We first show that if $s=d$ then for any $t\geq 0$, $\reg\sI^t\geq dt$. This is because otherwise suppose there is a number $t$ such that $\reg \sI^t=dt-a$ for some $a>0$. Then $\sI^t(dt-a)$ is generated by its global sections. Thus the line bundle $\sO_W((dt-a)\mu^* H-tE)$ is nef. That means $s\leq (dt-a)/t<d$ which is contradict to $s=d$.

Now let $\reg \sI^t=dt+e_t$ for $t\geq 0$. Note that from the above argument $e_t\geq 0$  and therefore $\liminf \{e_t\}\geq 0$. Also notice that $e\neq-\infty$ since $s=d$ and by Proposition \ref{pro:15}. Our plan is to prove the following inequalities
$$e\leq \liminf \{e_t\}\leq \limsup \{e_t\}\leq e.$$

First we show $\limsup \{e_t\}\leq e$. Since $e=b_1(\sO_W)$ there is a number $l$ such that $\textbf{m}=(e,l)$ and $\sO_W$ is $\textbf{m}$-regular. Then $\sO_W$ is $\textbf{m}+\nN^2$-regular. This implies that
$$H^i(Y,\sO_W(e-i,t))=0\quad\mbox{ for } t\gg 0.$$
Recall that for $t$ large enough, one has $H^i(Y,\sO_W(e-i,t))=H^i(\nP^n,\sI^t(dt+e-i))$. Thus we have the vanishing of $H^i(\nP^n,\sI^t(dt+e-i))=0$ for $t$ sufficiently large and therefore $\reg \sI^t\leq dt+e$. Thus $\limsup \{e_t\}\leq e$.

Next we show that $e\leq \liminf \{e_t\}$. Let $r=\liminf \{e_t\}$. Then there is a sequence $\{e_{t_j}\}^{\infty}_{j=0}$ such that $e_{t_j}=r$ for $t_j\rightarrow +\infty$. Thus by $H^i(Y,\sO_W(e_{t_j}-i,t_j))=H^i(\nP^n,\sI^{t_j}(dt_j+e_{t_j}-i))$ again, one has
$$H^i(Y,\sO_W(r-i,t_j))=0\quad\mbox{ for } j\gg 0,$$
which implies that $R^iq_*\sO_W(r-i)=0$. Thus $\sO_W$ is relative $r$-regular with respect to $q$. Since $e=\reg_q\sO_W$, we see that $e\leq r$ and therefore $e\leq \liminf \{e_t\}$.

Hence we get $e=\lim e_t$ and therefore $\reg \sI^t=dt +e$ for $t$ large enough.
\end{proof}

Immediately, we have the following corollary describing a typical type of variety having linear asymptotic regularity, which is not so obvious.

\begin{corollary} Suppose that $\sI$ defines a variety cut out by quadrics, then its asymptotic regularity is linear, i.e., for $t\gg 0$, one has
$$\reg \sI^t=2t+e,$$
for some positive integer $e$.
\end{corollary}
\begin{proof} Let $X$ be the variety defined by $\sI$. Since $X$ is cut out by quadrics, we see $s\leq 2$, where $s$ is the $s$-invariant of $\sI$. Also a generic secant line will cut $X$ by two distinct points. This shows that $s\geq 2$. Thus $s=2$ and then the result follows from Theorem  \ref{thm:02}.
\end{proof}

\begin{remark} (1) Note that the constant $e$ can be computed as the relative regularity of $\sO_W$. But it still hard to obtain in practise.

(2) From the corollary above, we see that these varieties having linear asymptotic regularity in fact are not so rare as we expected, although their Rees algebra $\oplus \sI^t$ are not finitely generated in general.

(3) In Chardin's work \cite{Chardin:Powersofideals}, he shows that if $I$ is a homogeneous ideal generated in a single degree $d$, then the limit
$$\lim_{t\rightarrow \infty} (\reg (I^t)^{sat}-dt)$$
exist but could be $-\infty$. Thus by using the $s$-invariant $s$, we give a clear picture for his result.
\end{remark}

In the light of Theorem \ref{thm:02}, it would be convenient to give a name for the constant $e$. We suggest the following definition.

\begin{definition} Suppose that an ideal sheaf $\sI$ satisfies $s(\sI)=d(\sI)$. The {\em asymptotic regularity constant} $e(\sI)$ is defined as the constant part of $\reg \sI^t=dt+e$ for $t$ large enough. We make a convention that if $s(\sI)<d(\sI)$, then $e(\sI)=-\infty$.
\end{definition}

\begin{remark}
According to Theorem \ref{thm:02}, the asymptotic regularity constant $e$ is always nonnegative if it is not $-\infty$. We hope that this number would carry some geometric or algebraic information of $\sI$. For example, we will see in Remark \ref{rmk:01}, it is determined by the surjectivity of a Guass map.
\end{remark}

We conclude this section by giving an example to compute a bound for $e$.


\begin{example} Still we assume that $s(\sI)=d(\sI)$. If we know the degrees of the generators of $W$ in $Y=\nP^n\times \nP(V)$, then we could give a bound for the asymptotic regularity constant $e$. Not surprisingly, such bound would be every large in general. Specifically, assume that $\sI_W$ is cut out by equations of  bidegree $(d_i,d'_i)$ for $i=1,\cdots, l$. Let $D=\max_i\{d_i\}$. Then for any point $y\in \nP(V)$, in the fiber $Y_y$, $W_y$ is cut out by equations of degree no more than $D$. Assume that $(A,m)=(\sO_{\nP(V),y},m_y)$ the local ring of $y$. Then for integer $k\geq 1$, in the thickening fiber $Y_{ky}=\Proj A/m^k[x_0,\cdots, x_n]$, the thickening $W_{ky}$ is still cut out by equations of degrees no more than $D$. According to the result of Chardin, Fall and Nagel \cite[Example 3.6]{Chardin:BoundRegModules}, the regularity of $W_{ny}$ is bounded by $(n+1)(D-1)+1$ if $n\leq 2$, and by $(3D^3)^{2^{n-3}}$ if $n\geq 3$. Then an easy application of formal function theorem will give us that
\begin{displaymath}
e \leq \left\{ \begin{array}{ll}
(n+1)(D-1) & \textrm{if $n\leq 2$,}\\
(3D^3)^{2^{n-3}}-1 & \textrm{if $n\geq 3$.}
\end{array} \right.
\end{displaymath}
However, it is very difficult in general to find the defining equations of $W$ in $Y$.

\end{example}

\section{Tensor algebra of conormal bundles}
\noindent In the rest of the paper, we study the regularity bounds for powers of an ideal sheaf under its geometric conditions. We always assume in the sequel that $X$ is a locally complete intersection of equidimension $n$ in the projective space $\nP$ (we omit the dimension of the projective space $\nP$, which does not come into our results), defined by an ideal sheaf $\sI$.  Denote by $N^*$ the conormal sheaf of $X$ which is $\sI/\sI^2$ by definition. Since $X$ is a locally complete intersection the conormal sheaf $N^*$ is then a locally free sheaf on $X$, i.e., the conormal bundle of $X$. We fix a minimal free resolution of $\sI$ as follows
\begin{equation}
\cdots\rightarrow\oplus \sO_{\nP}(-d_{3,j})\rightarrow\oplus \sO_{\nP}(-d_{2,j})\rightarrow\oplus \sO_{\nP}(-d_{1,j})\rightarrow\oplus \sO_{\nP}(-d_{0,j})\rightarrow\sI\rightarrow 0.
\end{equation}
Tensoring the minimal resolution with the structure sheaf $\sO_X$, we obtain a complex
\begin{equation}\label{eq:01}
\cdots\rightarrow\oplus \sO_X(-d_{3,j})\rightarrow\oplus \sO_X(-d_{2,j})\rightarrow\oplus \sO_X(-d_{1,j})\rightarrow\oplus \sO_X(-d_{0,j})\rightarrow N^*\rightarrow 0,
\end{equation}
which has homology sheaves
\begin{equation}
\cdots,\quad \sH_3=\wedge^4N^*,\quad \sH_2=\wedge^3N^*,\quad \sH_1=\wedge^2N^*,\quad \sH_0=0,
\end{equation}
i.e., $$\sH_i=\wedge^{i+1}N^*\quad\mbox{ for }i\geq 1.$$
The complex (\ref{eq:01}) can be viewed as a non-exact resolution of the conormal bundle $N^*$. Our strategy is to deduce a vanishing theorem of the tensor algebra of $N^*$ from this complex.

For any number $p\geq 0$, we say that $\sI$ is $p$-th partial $m$-regular if
$$d_{i,j}-i\leq m,\quad\mbox{ for }\ 0\leq i\leq p.$$
Obviously if $\sI$ is $(p+1)$-th partial $m$-regular then it is $p$-th partial $m$-regular.
We also assume that the structure sheaf $\sO_X$ is $r_X$-regular. In order to avoid some trivial situation, we always assume that the regularity and the partial regularities of $\sI$ are at least $2$ (in fact this just means that $X$ is not a linear space).

The following lemma is easy to prove and can be found in \cite[Lemma 2.10]{Ein:SyzygyKoszul}.
\begin{lemma}\label{pro:31} Let
$$Q_\bullet: \cdots \rightarrow Q_2\rightarrow Q_1\stackrel{\epsilon}{\rightarrow} Q_0\rightarrow 0$$
be a complex of coherent sheaves on $X$, with $\epsilon$ surjective. Assume that
\begin{enumerate}
\item [(1)] $H^k(X,Q_1)=H^{k+1}(X,Q_2)=\cdots=H^n(X,Q_{n-k+1})=0$;
\item [(2)] $H^{k+1}(X,\sH_1(Q_\bullet))=H^{k+2}(X,\sH_2(Q_\bullet))=\cdots=H^n(X,\sH_{n-k}(Q_\bullet))=0$.
\end{enumerate}
Then $H^k(X,Q_0)=0$.
\end{lemma}

The theorem we established in this section is the following vanishing theorem of the tensor products of $N^*$.

\begin{theorem}\label{thm:01} Assume that $\sI$ is $p$-th partial $r_p$-regular for $p\leq n-1$ and $\sO_X$ is $r_X$-regular. Then one has
\begin{align*}
&H^n(X,T^aN^*(k))=0 &&\mbox{ for }k\geq ar_p+r_X-n,\\
&H^{n-1}(X,T^aN^*(k))=0  &&\mbox{ for }k\geq ar_p+1+r_X-n,\\
&H^{n-2}(X,T^aN^*(k))=0 &&\mbox{ for }k\geq (a+1)r_p+r_X-n,\\
&H^{n-3}(X,T^aN^*(k))=0 &&\mbox{ for }k\geq (a+2)r_p+r_X-n,\\
&\cdots && \cdots,\\
&H^{n-p}(X,T^aN^*(k))=0 &&\mbox{ for }k\geq (a+p-1)r_p+r_X-n.
\end{align*}
\end{theorem}
\begin{proof} We first establish the vanishing of $H^n$ groups. From the complex (\ref{eq:01}) and by Lemma \ref{pro:31}, we see that $$H^n(X,N^*(k))=0,\quad \mbox{for } k\geq r_p+r_X-n.$$ Tensoring $N^*$ to the the complex (\ref{eq:01}), we obtain a complex
$$\cdots\rightarrow \oplus N^*(-d_{0,j})\rightarrow T^2N^*\rightarrow 0.$$ Thus immediately by Lemma \ref{pro:31}, we have
$$H^n(X,T^2N^*(k))=0,\quad \mbox{for } k\geq 2r_p+r_X-n.$$ Repeating such procedure we then get $H^n(X,T^aN^*(k))=0$ for $k\geq ar_p+r_X-n$.

Now we establish the vanishing of $H^{n-1}$ groups. In the complex (\ref{eq:01}) notice that $$H^{n-1}(X,\sO_X(k-d_{0,j}))=H^n(X,\sO_X(k-d_{1,j}))=0,\quad\mbox{for } k\geq r_p+1+r_X-n.$$
Then by Lemma \ref{pro:31}, we obtain $H^{n-1}(X,N^*(k))=0$, for $k\geq r_p+1+r_X-n$.
Tensoring $N^*$ to the complex (\ref{eq:01}), we have a complex
$$\cdots\rightarrow \oplus N^*(-d_{1,j})\rightarrow \oplus N^*(-d_{0,j})\rightarrow T^2N^*\rightarrow 0,$$
which has homology sheaf $\sH_0=0$. Notice that we have established
$$H^{n-1}(X,N^*(k-d_{0,j}))=H^n(X,N^*(k-d_{1,j})),\quad \mbox{for }k\geq 2r_p+1+r_X-n.$$
Then by Lemma \ref{pro:31}, we obtain the vanishing of $H^{n-1}(X,T^2N^*(k))$. Repeatedly we then obtain  $H^{n-1}(X,T^aN^*(k))=0$ for $k\geq ar_p+1+r_X-n$.

Recall that for any nonnegative number $a$ and $b$, $\wedge^aN^*\otimes T^bN^*$ is a direct summand of $T^{a+b}N^*$ since $N^*$ is locally free and we work over $\nC$. Thus the vanishing of the cohomology groups of $T^{a+b}N^*$ will automatically give the vanishing of the cohomology groups of $\wedge^aN^*\otimes T^bN^*$. We will use this fact repeatedly.

In order to establish the vanishing of $H^{n-2}(X,T^aN^*(k))$ for $a\geq 1$, we tensor $T^{a-1}N^*$ to the complex (\ref{eq:01}) to get a complex
$$\cdots\rightarrow \oplus T^{a-1}N^*(-d_{2,j})\rightarrow \oplus T^{a-1}N^*(-d_{1,j})\rightarrow \oplus T^{a-1}N^*(-d_{0,j})\rightarrow T^aN^*\rightarrow 0,$$
which has homology sheaves $\sH_0=0$ and
$$\sH_i=\wedge^{i+1}N^*\otimes T^{a-1}N^*,\quad \mbox{for }i\geq 1.$$
Since inductively, we have established the vanishing of
$$H^{n-2}(X,T^{a-1}N(k-d_{0,j})), H^{n-1}(X,T^{a-1}N^*(k-d_{1,j})), H^n(X,T^{a-1}N^*(k-d_{2,j}))$$
and the vanishing of $H^n(X,\wedge^2N^*\otimes T^{a-1}N^*(k))$, which from the vanishing of $H^n(X,T^{a+1}N^*(k))$. Thus by Lemma \ref{pro:31} again, we have $H^{n-2}(X,T^aN^*(k))=0$ for $k\geq (a+1)r_p+r_X-n$.

Inductively, we can finally prove the theorem.
\end{proof}

In order to give regularity bounds for powers of $\sI$, we inductively consider the exact sequence for any $a\geq 1$,
$$0\longrightarrow \sI^{a+1}\longrightarrow \sI^a\longrightarrow S^aN^*\longrightarrow 0.$$ For any $k\in \nZ$, we denote the morphism
\begin{equation}\label{eq:02}
\phi_a:H^0(\sI^a(k))\longrightarrow H^0(S^aN^*(k)).
\end{equation}
The morphism $\phi_a$ actually depends on the twist $k$ but we omit it by abuse of notation. The crucial point is to prove the surjectivity of $\phi_a$ for specific twisting $k$. For this we will put $\phi_a$ in a commutative diagram and to analyze each morphisms in it.

We tensor $\sO_X$ to the morphism $\oplus \sO_{\nP}(-d_{0,j})\rightarrow \sI\rightarrow 0$ of the first piece in the complex (\ref{eq:01}) to get the following diagram
\begin{equation}\label{eq:03}
\begin{CD}
\oplus \sO_{\nP}(-d_{0,j}) @>>> \sI@>>> 0\\
@VVV @VVV\\
\oplus\sO_X(-d_{0,j})@>>> N^*@>>> 0
\end{CD}
\end{equation}
Now for any $k,s\in \nZ$, we have the following morphisms on global sections
$$\begin{CD}
\oplus H^0(\sO_{\nP}(k-d_{0,j}) @>>> H^0(\sI(k))\\
@VVV @VVV\\
\oplus H^0(\sO_X(k-d_{0,j}))@>>> H^0(N^*(k))
\end{CD},\mbox{and}
\quad\quad\quad
\begin{CD}
H^0(\sI^a(s))\\
@VV\phi_aV \\
H^0(S^aN^*(s))
\end{CD}
$$
We tensor the left-hand-side diagram to the right-hand-side one above to deduce the following diagram and mark morphisms in it,
$$
\xymatrix{
 \oplus H^0(\sO_{\nP}(k-d_{0,j})\otimes H^0(\sI^a(s)) \ar[d]^{1\otimes\phi_a} \ar[r]& H^0(\sI(k))\otimes H^0(\sI^a(s)) \ar[dd]\ar[r] & H^0(\sI^{a+1}(k+s)) \ar[dd]^{\phi_{a+1}} \\
 \oplus H^0(\sO_{\nP}(k-d_{0,j}))\otimes H^0(S^aN^*(s)) \ar[d]^{c_a}&\\
 \oplus H^0(S^aN^*(k-d_{0,j}+s))\ar[r]^{w_a} & H^0(N^*\otimes S^aN^*(k+s)) \ar[r]^{u_{a+1}} & H^0(S^{a+1}N^*(k+s))}
$$
We simply write it as follows
$$
\xymatrix{
 \oplus H^0(\sO_{\nP}(k-d_{0,j})\otimes H^0(\sI^a(s)) \ar[d]^{1\otimes\phi_a} \ar[rr]&& H^0(\sI^{a+1}(k+s)) \ar[dd]^{\phi_{a+1}} \\
 \oplus H^0(\sO_{\nP}(k-d_{0,j}))\otimes H^0(S^aN^*(s)) \ar[d]^{c_a}&\\
 \oplus H^0(S^aN^*(k-d_{0,j}+s))\ar[r]^{w_a} & H^0(N^*\otimes S^aN^*(k+s)) \ar[r]^{u_{a+1}} & H^0(S^{a+1}N^*(k+s))}
$$
Thus inductively, in order to build the surjectivity of $\phi_{a+1}$ it is enough to build the surjectivity of $\phi_a$, $c_a$, $w_a$ and $u_{a+1}$.

Since the formulas in Theorem \ref{thm:01} depend on the dimension of $X$ we consider lower and higher dimension cases for $X$ in the following two sections.

\section{Higher dimensional varieties}
\noindent Keep notation as in Section 3 and assume that $n=\dim X\geq 3$. Assume also that $\sI$ is $n$-th partial $r_n$-regular (recall that we always require that $r_n\geq 2$), the structure sheaf $\sO_X$ is $r_X$-regular and $X$ is $n_0$-normal, i.e., $H^1(\nP,\sI(k))=0$ for $k\geq n_0$. We start by giving the regularity bounds for $T^aN^*$ which follows easily from Theorem \ref{thm:01}.

\begin{proposition}\label{pro:01}For $a\geq 1$ one has
    $$T^aN^* \mbox{ is } (a+n-2)r_n+r_X-n+1\mbox{ regular}.$$
\end{proposition}
\begin{proof} It is an immediate consequence of Theorem \ref{thm:01}.
\end{proof}

\begin{lemma}\label{pro:03} For any $a\geq 1$ the morphism
$$u_a:H^0(N^*\otimes S^{a-1}N^*(k))\longrightarrow H^0(S^aN^*(k))$$
is surjective if $k\geq (a+n-2)r_n+r_X-n$.
\end{lemma}
\begin{proof} In the exact sequence
$$\cdots\longrightarrow S^{a-2}N^*\otimes \wedge^2N^*\longrightarrow S^{a-1}N^*\otimes N^*\stackrel{\partial}{\longrightarrow} S^aN^*\longrightarrow 0.$$
each term is $(a+n-2)r_n+r_X-n+1$ regular by Proposition \ref{pro:01}. Let $\sK=\ker \partial$, then by chasing from this complex, it is easy to see that $H^1(\nP,\sK(a+n-2)r_n+r_X-n)=0$. Thus the surjectivity of $u_a$ follows immediately.
\end{proof}

\begin{lemma}\label{pro:02} For $a\geq 0$ the morphism
$$w_a:\oplus H^0(S^aN^*(k-d_{0,j}))\longrightarrow H^0(N^*\otimes S^aN^*(k))$$
is surjective for $k\geq (n+a)r_n+r_X-n$.
\end{lemma}
\begin{proof} Tensor $S^aN^*$ to the complex (\ref{eq:01}) to get the complex
$$
\cdots\rightarrow\oplus S^aN^*(-d_{2,j})\rightarrow\oplus S^aN^*(-d_{1,j})\rightarrow\oplus S^aN^*(-d_{0,j})\stackrel{\delta}{\rightarrow} N^*\otimes S^aN^*\rightarrow 0,
$$
The homology sheaves of this complex are  $$\sH_i=\wedge^{i+1}N^*\otimes S^aN^*\quad\quad \mbox{ for } i\geq 1$$
and $\sH_0=0$.
Let $\sK$ be the kernel sheaf of the morphism $\delta$, then we have a complex
$$\cdots\rightarrow\oplus S^aN^*(-d_{2,j})\rightarrow\oplus S^aN^*(-d_{1,j})\rightarrow \sK\rightarrow 0$$
which is exact at $\sK$. Thus it is enough to show that for $k\geq (n+a)r_n+r_X-n$ we have $H^1(\sK(k))=0$. Since from Proposition \ref{pro:01}, we see that for $k\geq (n+a)r_n+r_X-n$,
$$H^1(S^aN^*(k-d_{1,j}))=H^2(S^aN^*(k-d_{2,j}))=\cdots=H^n(S^aN^*(k-d_{n,j}))=0,$$
and
$$H^2(\wedge^2N^*\otimes S^aN^*(k))=H^3(\wedge^3N^*\otimes S^aN^*(k))=\cdots=H^n(\wedge^nN^*\otimes S^aN^*(k))=0.$$
Thus by Lemma \ref{pro:31}, our result follows.
\end{proof}

\begin{proposition}\label{pro:04} For $a\geq 1$, the morphism
    $$\phi_a:H^0(\sI^a(k))\longrightarrow H^0(S^aN^*(k))$$
    is surjective for $k\geq ar_n+\max(n_0,(n-1)r_n+r_X-n)$.
\end{proposition}
\begin{proof}We prove by induction on $a$. We start by showing the surjectivity of the morphism $\phi_1: H^0(\sI(k))\rightarrow H^0(N^*(k))$. Twisting the diagram (\ref{eq:03}) by $\sO_{\nP}(k)$ and then taking global sections to obtain the diagram
$$\begin{CD}
\oplus H^0(\sO_{\nP}(k-d_{0,j})) @>>> H^0(\sI(k))\\
@Vc_0VV @VVV\\
\oplus H^0(\sO_X(k-d_{0,j}))@>w_0>> H^0(N^*(k))
\end{CD}$$
Then observe that
\begin{enumerate}
\item The surjectivity of $c_0$ is guaranteed by the normality of $X$ when $k\geq r_n+n_0$
\item The surjectivity of $w_0$ has been built in Lemma \ref{pro:02} when $k\geq nr_n+r_X-n$.
\end{enumerate}
Write $s_1=r_n+\max(n_0,(n-1)r_n+r_X-n)$ then $\phi_1$ is surjective if $k\geq s_1$. Note that $N^*$ is also $s_1$ regular by Proposition \ref{pro:01}.

Inductively suppose that $\phi_a$ is surjective. We write $s_a=ar_n+\max(n_0,(n-1)r_n+r_X-n)$ and note that $S^aN^*$ is $s_a$ regular by Proposition \ref{pro:01}. we show the surjectivity of $\phi_{a+1}$. For this consider the diagram constructed in Section 3
$$
\xymatrix{
 \oplus H^0(\sO_{\nP}(k'-d_{0,j})\otimes H^0(\sI^a(s_a)) \ar[d]^{1\otimes\phi_a} \ar[rr]&& H^0(\sI^{a+1}(k'+s_a)) \ar[dd]^{\phi_{a+1}} \\
 \oplus H^0(\sO_{\nP}(k'-d_{0,j}))\otimes H^0(S^aN^*(s_a)) \ar[d]^{c_a}&\\
 \oplus H^0(S^aN^*(k'-d_{0,j}+s_a))\ar[r]^{w_a} & H^0(N^*\otimes S^aN^*(k'+s_a)) \ar[r]^{u_{a+1}} & H^0(S^{a+1}N^*(k'+s_a))}
$$
Observe that
\begin{enumerate}
\item $1\otimes\phi_a$ is surjective since so is $\phi_a$.
\item $c_a$ is surjective if $k'\geq r_n$ since $S^aN^*$ is $s_a$-regular.
\item $w_2$ is surjective if $k'+s_a\geq (n+2)r_n+r_X-n$ by Lemma \ref{pro:02}.
\item $u_{a+1}$ is surjective if $k'+s_a\geq (a+n-1)r_n+r_X-n$ by Lemma \ref{pro:03}.
\end{enumerate}
Thus if write $s_{a+1}=(a+1)r_n+\max(n_0,(n-1)r_n+r_X-n)$ then $\phi_{a+1}$ is surjective if $k\geq s_{a+1}$.

\end{proof}

As an application of above results, we first consider the case that $X$ is a nonsingular projective variety embedded by an adjoint line bundle
$$L_d=K_X+dA+B,$$
where $A$ is a very ample line bundle, $B$ is nef line bundle and $K_X$ is the canonical line bundle. According to the work of Ein and Lazarsfeld \cite{Ein:SyzygyKoszul}, if $d\geq \dim X+1+p$ then $X$ satisfies Property $N_p$. For $L_d$ positive enough, we have a chance to get regularity bounds for powers of $\sI$.
\begin{proposition} As setting above. Assume that $d\geq 2(\dim X+1)$. Then for any $a\geq 1$, one has
$$\sI^a \mbox{ is } 2a+2n-2 \mbox{ regular}.$$
\end{proposition}
\begin{proof} Inductively using Proposition \ref{pro:04}, Proposition \ref{pro:01} and exact sequences
$$0\longrightarrow \sI^{a+1}\longrightarrow \sI^a\longrightarrow S^aN^*\longrightarrow 0,$$
the result follows immediately.
\end{proof}

As another application, we next consider the case that $X$ is a locally complete intersection and $\sI$ is $r$-regular. Then we are able to give the regularity bounds for $\sI^a$ in terms of $r$.

\begin{proposition} As setting above, then for any $a\geq 1$ one has
$$\sI^a \mbox{ is } ar+\max(r,(n-1)r-n) \mbox{ regular}.$$
\end{proposition}

\section{Curves and Surfaces}

\noindent In this section, we follow the same approach as previous section to study the case of curves and surfaces. Since the argument is exactly same as previous we shall be brief.

Let us first start with the case of curves.

\begin{proposition} Keep notation as in the beginning of Section 3 and assume that $\dim X=1$.
\begin{enumerate}
\item Assume that $\sI$ is $0$-th partial $r_0$-regular and $\sO_X$ is $r_X$-regular. Then
    $$T^aN^* \mbox{ is } (ar_0+r_X)\mbox{-regular}.$$
\item Assume further that $\sI$ is $1$-st partial $r_1$-regular and $X$ is $n_0$-normal, then for $a\geq 1$, the morphism
    $$\phi_a:H^0(\sI^a(k))\longrightarrow H^0(S^aN^*(k))$$
    is surjective for $k\geq ar_1+\max(n_0,r_X)$.
\end{enumerate}
\end{proposition}
\begin{proof} (1) is an immediate corollary of Theorem \ref{thm:01}. (2) is the same argument as Proposition \ref{pro:04}.
\end{proof}

In fact for a nonsingular projective curve embedded by a large degree line bundle $L$, Vermeire \cite{Vermeire:RegPowers} has used the same idea to prove the following theorem.

\begin{proposition} Assume that $X$ is a nonsingular projective curve of genus $g$ embedded by a line bundle $L$ with $\deg L\geq 2g+3$. Then for any $a\geq 1$,
$$\sI^a \mbox{ is } (2a+1)\mbox{-regular}.$$
\end{proposition}
\begin{proof} Just apply the proposition above.
\end{proof}

The syzygies of the curve embedded by such line bundle $L$ has been studied by Green and Lazarsfeld. And according to their theorem such curve satisfies at leas Property $(N_2)$.

Also in his work, Vermeire has proved that $\sI^a$ is $2a$-regular if and only if the Guass map $\Phi_L$ is surjective. Plus the linearity theorem we established in Section 2, we are able to give the effective value for $\reg \sI^a$, which slightly strengthens Vermerie's result.

\begin{proposition}\label{pro:05} Assume that $X$ is a nonsingular projective curve of genus $g$ embedded by a line bundle $L$ with $\deg L\geq 2g+3$. Then for any $a\geq 1$,
\begin{enumerate}
\item $\reg\sI^a=2a+1$ if and only if the Guass map $\Phi_L$ is not surjective
\item $\reg \sI^a=2a$ if and only if the Guass map $\Phi_L$ is surjective.
\end{enumerate}
\end{proposition}

\begin{remark}\label{rmk:01} Note that the proposition means that if for some $a$ $\reg \sI^a=2a$ then for any $a$, $\reg \sI^a=2a$. Thus the asymptotic behavior of the regularity of curves in this case is quite clear: the asymptotic regularity constant $e$ is either $0$ or $1$, which depends on the surjectivity of the Guass map $\Phi_L$.
\end{remark}

Next, we give regularity bounds for any powers of $\sI$ by assuming that we know the regularity of $\sI$.

\begin{proposition} Assume that $X$ is a local complete intersection of dimension one in $\nP$ defined by an ideal sheaf $\sI$ which is $r$-regular, then for any $a\geq 1$, one has
$$\sI^a \mbox{ is } ar\mbox{-regular}.$$
\end{proposition}

\begin{remark} This formula is sharp. For example assume that $X$ is a rational normal curve, then $\reg\sI=2$ and therefore $\reg \sI^a=2a$, which is sharp.
\end{remark}

Now we turn to the case of surfaces.
\begin{proposition} Keep notation as in the beginning of Section 3 and assume that $\dim X=2$.
\begin{enumerate}
\item Assume that $\sI$ is $1$-th partial $r_1$-regular and $\sO_X$ is $r_X$-regular. Then
    $$T^aN^* \mbox{ is } (ar_1+r_X)\mbox{-regular}.$$
\item Assume further that $\sI$ is $2$-nd partial $r_2$-regular and $X$ is $n_0$-normal, then for $a\geq 1$, the morphism
    $$\phi_a:H^0(\sI^a(k))\longrightarrow H^0(S^aN^*(k))$$
    is surjective for $k\geq ar_2+\max(n_0,r_2+r_X-2)$.
\end{enumerate}
\end{proposition}

As in the higher dimensional case, here we consider a nonsingular surface embedded by an adjoint line bundle. This can be viewed as a generalization of Vermier's result.

\begin{proposition} Assume that $X$ is a nonsingular projetive surface embedded by an adjoint line bundle
$L_d=K_X+dA+B$
where $A$ is a very ample line bundle and $B$ is nef line bundle and $K_X$ is the canonical line bundle. Assume that $d\geq \dim X+1+3$. Then for any $a\geq 1$, one has
$$\sI^a \mbox{ is } (2a+2) \mbox{-regular}.$$
\end{proposition}

\begin{remark} It would be very interesting to find the effective value of $\reg\sI^a$. We hope that there is a similar result as in the case of curves in Proposition \ref{pro:05}.
\end{remark}

Let us conclude this section by the following proposition, which
gives a bound for the regularity of $\sI^a$ in terms of the regularity of $\sI$.

\begin{proposition} Assume that $X$ is a local complete intersection of dimension $2$ in a projective space defined by an ideal sheaf $\sI$ which is $r$-regular, then for any $a\geq 1$, one has
$$\sI^a \mbox{ is } (ar+r-2) \mbox{-regular}.$$
\end{proposition}

\begin{remark} Again this formula is sharp. For example if $X$ is a Veronese surface, then it has the minimal degree and therefore has the minimal regularity $2$. In this situation, we see that $\reg\sI^a=2a$.
\end{remark}

\end{document}